\documentclass[letterpaper, 10 pt, conference]{ieeeconf}  

\IEEEoverridecommandlockouts                              
\usepackage{graphics} 
\usepackage{epsfig} 
\usepackage{times} 
\usepackage{amsmath} 
\usepackage{amssymb}  
\usepackage{ bbold }
\usepackage{mathtools}
\mathtoolsset{showonlyrefs}
\usepackage{color}
\usepackage{cite}
\usepackage{footmisc}
\usepackage{hyperref}

\title{\LARGE \bf
Discrete-Time Linear-Quadratic Regulation via Optimal Transport
}

\newtheorem{lemma}{Lemma}
\newtheorem{theorem}{Theorem}

\newtheorem{remark}{Remark}

\newcommand{\bmpi}{\boldsymbol{\pi}}
\newcommand{\bmrho}{\boldsymbol{\rho}}
\newcommand{\bmtau}{\boldsymbol{\tau}}

\author{Mathias Hudoba de Badyn, Erik Miehling, Dylan Janak, Beh\c{c}et A\c{c}{\i}kme\c{s}e, \\Mehran Mesbahi, Tamer Ba\c{s}ar, John Lygeros, Roy S.~Smith
\thanks{This  work  was  supported  by the  ETH  Foundation, and the SNSF under NCCR Automation. 
MHdB performed part of this work during a stay at the Erwin Schr\"{o}dinger Institute Optimal Transport workshop in May 2019. 
EM and TB are funded in part by US Army Research Laboratory Cooperative Agreement W911NF-17-2-0196, and in part by the Air Force Office of Scientific Research (AFOSR) Grant FA9550-19-1-0353.
MM is  funded in part by the Air Force Office of Scientific Research grant FA9550-16-1-0022.}
\thanks{MHdB, JL, \& RSS are with the Automatic Control Laboratory at ETH Z\"{u}rich, Switzerland.
        {\tt\small \{mbadyn, jlygeros, rsmith \}@control.ee.ethz.ch}}%
\thanks{EM \& TB are with the Coordinated Science Laboratory at the University of Illinois at Urbana-Champaign, Urbana, USA.
        {\tt\small \{miehling, basar1\}@illinois.edu}}%
\thanks{DJ, MM \& BA are with the William E.~Boeing Department of Aeronautics and Astronautics at the University of Washington, Seattle, USA.
        {\tt\small \{dj137, mesbahi, behcet\}@uw.edu}}%
}

\begin{document}

\maketitle
\thispagestyle{empty}
\pagestyle{empty}

\begin{abstract}
In this paper, we consider a discrete-time stochastic control problem with uncertain initial and target states.
We first discuss the connection between optimal transport and stochastic control problems of this form. Next, we formulate a linear-quadratic regulator problem where the initial and terminal states are distributed according to specified probability densities. A closed-form solution for the optimal transport map in the case of linear-time varying systems is derived, along with an algorithm for computing the optimal map. Two numerical examples pertaining to swarm deployment demonstrate the practical applicability of the model, and performance of the numerical method.
\end{abstract}

\section{Introduction}

The problem of steering the states of a linear system from an initial distribution to a terminal distribution has attracted much interest in recent years~\cite{Chen2015a,Chen2015,Chen2018c}.
Applications of such controllers include the density control of swarms~\cite{Eren2017,nazlipaper} and networked dynamical systems~\cite{Hudobadebadyn2018}, as well as opinion dynamics~\cite{Albi2016}.

Optimal mass transport is a mathematical framework for deriving mass-preserving maps between specified distributions that minimize a cost of transport.
The optimal transport cost, in some specific contexts called the \emph{Wasserstein metric}, provides a useful metric on the space of probability distributions.
This has been employed in a wide variety of fields, such as economics~\cite{galichon2018optimal}, machine learning~\cite{Frogner2015,Rolet2016}, computer vision~\cite{peyrechap}, and image processing~\cite{Rabin2011}.
The Wasserstein metric also allows one to tractably compute worst-case distributions in optimization problems~\cite{paymandist}, which have been applied in areas such as state estimation~\cite{ShafieezadehAbadeh2018}, and machine learning~\cite{kuhndata}.
The computation of the Wasserstein metric and corresponding transport map has also attracted much attention, in particular techniques allowing for computational speedup such as entropic regularization and Sinkhorn scaling~\cite{Cuturi2013,Peyre2015}.

The connection of optimal transport to continous-time control began with the seminal reformulation of optimal transport as a PDE-based fluid dynamics optimization problem~\cite{Benamou2000}.
In this approach, a velocity field is computed that minimizes the average kinetic energy of a fluid moving from one density to another.
Equivalently, this approach can be thought of as a single-integrator particle moving from an initial state with uncertainty described by an initial distribution, to a final state with an uncertainty described by a final distribution.
The cases of general linear time-varying (LTV) systems, and general LTV systems driven by noise (so-called \emph{Schr\"{o}dinger bridges}) were developed by~\cite{Chen2017}.

The latter paper~\cite{Chen2017} employs a Lagrangian-based cost function, where the static quadratic cost is replaced with a time-varying cost with dynamical constraints.
Such techniques were developed in~\cite{Agrachev2009}, which dealt with optimal transport with nonholonomic constraints.
In a similar problem configuration, the existence and uniqueness of transport maps were determined for linear--quadratic costs by \cite{Hindawi2011}. 
Other works include distributed optimal transport for swarms of single-integrators~\cite{KrishnanArxiv,Krishnan2018}, 
Perron-Frobenius operator methods for computing optimal transport over nonlinear systems~\cite{bermanpaper}, and a related problem regarding the steering of an LTV systems to a terminal state with specified expected value and covariance~\cite{panospaper,BakolasCDC,bakolas2018finite}.

While much attention has been paid to optimal transport of dynamical systems in continuous-time, there has been a marked lack of works discussing the implementation of such controllers in discrete time, which is a gap in the literature that needs to be addressed before optimal transport techniques can be implemented on digital controllers.
One contribution of this paper is to provide a rigorous analysis of the optimal transport problem for linear-quadratic regulation of LTV systems in discrete time.

In the present work, we discuss the theory and implementation of optimal transport for discrete-time linear-quadratic regulation for LTV systems.
Our contributions are as follows.
We formalize a previously-developed method of applying optimal transport methods to control by converting a class of optimal control problems to an optimal transport problem where the cost function is the optimal cost-to-go from an initial state to a terminal state.
This formalism is then applied to derive the closed-form solution of the discrete-time LQR problem with state-density constraints.
This problem is solved numerically, and the solution is then implemented on an example involving swarm deployment.

The paper is organized as follows.
We outline the notation and preliminaries on optimal transport in \S\ref{sec:math-prel}.
Our problem statement is outlined in \S\ref{sec:problem-statement}, where we discuss formulating optimal transport problems for control systems in terms of value functions.
Our results concerning optimal transport for LQR and its numerical computation are in \S\ref{sec:results}.
We present numerical examples and an application to swarm deployment in \S\ref{sec:examples}, and conclude the paper in \S\ref{sec:conclusion}.

\section{Mathematical Preliminaries}
\label{sec:math-prel}

In this section, we outline the notation used in the paper, as well as the necessary preliminaries on optimal transport.

\subsection{Notation}

The $n$-dimensional space of real numbers, non-negative real numbers, and positive real numbers are respectively denoted by $\mathbb{R}^n,\mathbb{R}^n_+$, and $\mathbb{R}^n_{++}$. 
We denote vectors in lower-case $x,y,z\in\mathbb{R}^n$, and matrices in capital-case $A,B,C\in\mathbb{R}^{n\times m}$.
Inequalities are interpreted component-wise.
Symmetric positive-definite and positive semi-definite cones of matrices are respectively denoted as $\mathcal{S}_{++}^n$ and $\mathcal{S}_{+}^n$.
For $Q \in \mathcal{S}_{+}^n$, we let $x^TQx = \|x\|_Q^2$.
The $n\times n$ identity matrix is denoted by $I_n$, or just $I$  if comformable dimensions are assumed.
$\mathbf{1}_n$ denotes the length-$n$ vector of all ones, and $\mathbf{0}$ denotes a matrix of zeros.
The identity map is denoted by $\mathrm{Id}(x) = x$.
The direct sum of $n$ $m\times m$ square matrices matrices $\{A_i\}_{i=1}^n$ is the $nm\times nm$ matrix formed by placing $A_1,\dots,A_n$ on the block diagonal. 
It is denoted by $\bigoplus_{i=1}^n A_i$.
The vectorization operation $U=\mathrm{vec}(\{u_k\}_{k=1}^n)$ denotes the vector $U \in\mathbb{R}^{nm}$ consisting of the stacked vectors $u_k\in\mathbb{R}^m$.

A measure space is a triple $(\mathcal{X},\mathcal{A},\mu)$ where $\mathcal{X}$ is a set, $\mathcal{A}$ is a $\sigma$-algebra on $\mathcal{X}$, and $\mu$ is measure on $(\mathcal{X},\mathcal{A})$.
We write a \emph{probability space} as $(\mathcal{X},\mu)$, where $\mu(\mathcal{X}) =1$ is a non-negative measure, and  we assume that $(\mathcal{X},\mu)$ is equipped with the Borel $\sigma$-algebra. 

For probability spaces $(\mathcal{X},\mu_0)$, $(\mathcal{Y},\mu_1)$, the \emph{pushforward map}, denoted by $\mu_1 := T_\# \mu_0$, is defined by the relation
\begin{align}
   \mu_1(B) = \mu_0(T^{-1}(B))
\end{align}
for each $B\in \mathcal{A}(Y)$.
If a random variable $x$ is distributed according to a probability density function $\rho(x)$, then we write $x\sim\rho(x)dx$.

\subsection{Optimal Transport}
In this section, we summarize four seminal forms of the optimal transport problem, and then specify the form of the optimal transport for our present work. 
One may consult excellent texts by Villani for a more in-depth discussion of the theory~\cite{Villani2009,Villani2003}.

Consider two probability spaces $(\mathcal{X},\mu_0)$ and  $(\mathcal{Y},\mu_1)$.
A \emph{transport map} $T:\mathcal{X} \rightarrow \mathcal{Y}$ is said to \emph{transport $\mu_0$ to $\mu_1$} if $T_\# \mu_0 = \mu_1$.
The Monge optimal transport problem seeks to find an optimal map $T$ that minimizes some cost of transport $c(x,T(x))$,
\begin{align}
    \begin{array}{ll}
      \inf_T & \int_{\mathcal{X}} c(x,T(x)) d\mu_0(x)\\
      \mathrm{s.t.} & T_\# \mu_0 = \mu_1.
    \end{array}\label{eq:8}\tag{OT1}
\end{align}
In general, if one of the measures $\mu_0,\mu_1$ has infinite second moment, then the cost of~\eqref{eq:8} may be infinite.
Furthermore, the pushforward constraint of \eqref{eq:8} makes this problem computationally intractable.
Kantorovich formulated a relaxation of \eqref{eq:8} that obtains the same minimizer under quadratic costs\footnote{The Kantorovich and Monge problems have corresponding minimizers under more general choices of $c(x,y)$, but we only consider the quadratic cost $c(x,y)= x^TQ_xx + y^TQ_yy + 2x^TQ_{xy}y$ in this paper.}, i.e., $c(x,y) = \frac{1}{2}\|x-y\|_2^2$.
The problem considers the set of joint probability distributions $\pi(x,y)$ on $\mathcal{X}\times\mathcal{Y}$ whose marginals are the initial and target measures,
\begin{align}
    \pi(A,\mathcal{Y}) = \mu_0(A),~\pi(\mathcal{X},B)= \mu_1(B),
\end{align}
for all Borel sets $A\subseteq\mathcal{X}$ and $B\subseteq\mathcal{Y}$.
With some abuse of notation, to make variables of operators (e.g., optimization, integration) we may write the above as
\begin{align}
    \pi(x,\cdot) = \mu_0(x),~\pi(\cdot,y) = \mu_1(y).
\end{align}
The Kantorovich optimal transport is then given by,
\begin{align}
    \begin{array}{ll}
      \inf_\pi & \int_{\mathcal{X}\times\mathcal{Y}} c(x,y) d\pi(x,y)\\
      \mathrm{s.t.} & \pi(x,\cdot) = \mu_0(x),~\pi(\cdot,y) = \mu_1(y).
    \end{array}\label{eq:9}\tag{OT2}
\end{align}
For the case of quadratic costs, \eqref{eq:9} obtains the same minimum as \eqref{eq:8}, and the optimal coupling satisfies $\pi^* = (\mathrm{Id}\times T^*)_{\#}\mu_0$, where $T^*(x)$ is the optimal map from \eqref{eq:8}.

The dual of \eqref{eq:9} has an explicit interpretation in economic theory of transport pricing~\cite{galichon2018optimal}, but perhaps more importantly, it offers insight into the structure of the optimal map $T$ in the case of quadratic costs.
For $\phi,\psi$ in the dual space of probability measures, the dual is given by,
\begin{align}
    \begin{array}{ll}
      \sup_{\phi,\psi} & \int_{\mathcal{X}} \phi(x)d\mu_0(x) - \int_{\mathcal{Y}} \psi(y) d\mu_1(y)\\
      \mathrm{s.t.} & \phi(x) - \psi(y) \leq c(x,y),~\forall (x,y)\in \mathcal{X}\times \mathcal{Y}.~~~~~
    \end{array}\label{eq:10}\tag{OT3}
\end{align}
When $c(x,y)= \frac{1}{2}\|x-y\|_2^2$, the optimal map $T^*(x)$ of \eqref{eq:8} can be written in terms of $\psi^*$ from \eqref{eq:10} as~\cite{Villani2003},
\begin{align}
    T^*(x) = \nabla\left(\frac{1}{2} x^Tx + \psi^*(x) \right),
\end{align}
and in particular it can be shown that $(\frac{1}{2} x^Tx + \psi^*(x))$ is a convex function~\cite{brenierfactorization}.
Note that in our notation, $\psi^*(x)$ refers to the \emph{optimal}  $\psi$, and not its Fenchel conjugate.

One final formulation of optimal transport we describe here is given by Brenier and Benamou in the form of an optimal control problem in a fluid dynamics setting.
Given initial and terminal densities $\rho_0,$ $\rho_1$, one seeks to find a smooth, time-dependent velocity field $v(x,t)$ taking $\rho_0$ to $\rho_1$ in unit time, while satisfying the continuity equation.
The velocity field minimizes the average kinetic energy of the fluid.
The problem is explicitly defined as~\cite{Benamou2000},
\begin{align}
    \begin{array}{ll}
      \sup_{\rho,v} & \int_0^1\int_{\mathbb{R}^n} \|v(x,t)\|_2^2\rho(x,t)dxdt\\
      \mathrm{s.t.} & \partial_t \rho(x,t) + \nabla \cdot(\rho(x,t)v(x,t)) = 0\\
                    & \rho(x,0) = \rho_0(x),~\rho(x,1) = \rho_1(x).
    \end{array}\label{eq:11}\tag{OT4}
\end{align}
In Lagrangian coordinates $X(x,t)$ with $X(x,0):=x$, $\partial_t X(x,t) = v(X(x,t),t)$, the solution to \eqref{eq:11} is given by a linear interpolation with the optimal map,
\begin{align}
    X(x,t) = x + t(T(x) - x)=: T_t(x),
\end{align}
and so the densities at time $t$ satisfy
\begin{align}
    \rho(x,t) := \rho_t(x) = (T_t)_{\#} \rho_0(x).
\end{align}

\section{Stochastic Optimal Control with State-Density Constraints}
\label{sec:problem-statement}

In this section, we consider an optimal transport approach for the discrete-time linear-quadratic regulator.
We present a formal discretization of the continuous-time controllers presented in \cite{Chen2017}, and extend this to the more general framework of LQR control.

We consider systems with a state $z_k\in\mathbb{R}^n$ of the form
\begin{equation}
    \begin{aligned}
      z_{k+1} &= A_kz_k + B_ku_k\\
      z_0 &\sim \rho_0(z)dz,
    \end{aligned} \label{eq:12}
\end{equation}
where the initial condition $z_0$ has some uncertainty described by a probability density $\rho_0(x)$ and $u_k\in \mathbb{R}^m$ is the control.
Our goal is to translate the system \eqref{eq:12} to a terminal state $z_{t_f}\sim \rho_1$ over a time horizon $0 \leq k \leq t_f$, where $\rho_1$ captures some desired uncertainty in the terminal state\footnote{As a technical assumption, we let $t_f \geq n$. This is not a constricting assumption, because OT problems do not in general scale well with $n$, and so we expect that in a real-world setting the OT methods in this paper would be applied to a reduced-order model (and hence small $n$) to compute references that would be tracked by a local, higher-fidelity controller.}.
The control should satisfy some optimality principle under an appropriate cost, and so an optimization problem with dynamics \eqref{eq:12} is,
\begin{align}
    \begin{array}{ll}
        \min_{u,\pi} & \mathbb{E}_{\pi}\left[ \sum_{k=0}^{t_f-1} c(z_k,u_k)\right]\\
      \mathrm{s.t.} & z_{k+1} = A_kz_k + B_ku_k\\
                 & z_0 \sim \rho_0(z)dz,~z_{t_f} \sim \rho_1(z)dz,
    \end{array}\label{eq:14}
\end{align}
where the expectation is with respect to a joint distribution $\pi(z_0,z_{t_f})$, as defined in \eqref{eq:9}.
The remark below formalizes a solution technique for problems of the form~\eqref{eq:14} which was used by \cite{Chen2017} to solve continuous-time optimal control problems with control costs.

\begin{remark}
\label{rem:1}
A general method to solve problems of the form \eqref{eq:14} is to first solve the deterministic problem
\begin{align}
   \left. \begin{array}{ll}
        \min_{u} &  \sum_{k=0}^{t_f-1} c(z_k,u_k)\\
      \mathrm{s.t.} & z_{k+1} = A_kz_k + B_ku_k\\
                 & z_0 = x,~z_{t_f} = y
    \end{array}\right\} = C(x,y)\label{eq:15},
\end{align}
to determine a formula $C(x,y)$ for the optimal cost-to-go from $x$ to $y$. 
Thus, \eqref{eq:14} can be re-written as
\begin{align}
  \begin{array}{ll}
        \min_{\pi} & \int_{\mathbb{R}^n\times\mathbb{R}^n} C(x,y) d\pi(x,y)\\
    \mathrm{s.t.} & \pi(x,\cdot) = \rho_0(x)dx,~\pi(\cdot,y) = \rho_1(y)dy,
  \end{array}\label{eq:17}
\end{align}
where the marginal constraints on $\pi$ encode the relevant distributions on the initial state $x$ and terminal state $y$.
Problem \eqref{eq:17} is clearly a Kantorovich optimal transport problem of the form \eqref{eq:9}, where the cost function is now the deterministic value function encoding the cost-to-go from $x$ to $y$, and the optimal coupling $\pi^*$ encodes a mapping between initial and terminal conditions $x$ and $y$.

The solution to \eqref{eq:17}, under appropriate assumptions on the cost $C(x,y)$, yields a coupling of the form
\begin{align}
    \pi^*(x,y) = (\mathrm{Id}\times T^*)_{\#} \mu_0(x),
\end{align}
where $y = T^*(x)$.
Finally, we note that  $\{u^*_k(x,T(x))\}_{k=1}^{t_f-1}$ solves \eqref{eq:14}.
\hfill\(\triangle\)
\end{remark}

\vspace{0.5em}
When $c(z_k,u_k) = (z_k - y)^TQ_k(z_k - y) + u_k^TR_k u_k$ for $y\sim \rho_1(z)dz$, we have the following LQR problem with stochastic initial and terminal constraints,
\begin{align}
    \begin{array}{ll}
        \min_{u} & \mathbb{E}\left[ \sum_{k=0}^{t_f-1} \left\{ \|z_k - y\|_{Q_k}^2 + \|u_k\|_{R_k}^2\right\}\right]\\
      \mathrm{s.t.} & z_{k+1} = A_kz_k + B_ku_k\\
                 & z_0 \sim \rho_0(z)dz,~z_{t_f} = y \sim \rho_1(z)dz.
    \end{array}\label{eq:16}
\end{align}
We solve this problem in the following section.

\section{Derivation of the Optimal Map}
\label{sec:results}

In this section, we outline the solutions to Problem~\eqref{eq:16}, beginning with the simplified case of a cost on the control only. 
Our main contribution in this section is the more-general LQR problem, outlined in~\ref{sec:discr-time-optim}.

\subsection{Discrete-Time Optimal Transport -- Control Cost Case}

Consider the task of transporting over a time horizon of $0\leq k \leq t_f$ a linear time-varying system with uncertain initial state $z_0$ characterized by  $\rho_0(z)$, to a final state $z_{t_f}$ with an uncertainty characterized by $\rho_1(z)$, with minimal control cost.
We assume that the dynamics $z_{k+1} = A_kz_k + B_k u_k$ are controllable over the interval $0\leq k \leq t_f$.
The problem is posed as
\begin{align}
  \begin{array}{ll}
    \min & \mathbb{E} \left[ \sum_{k=0}^{t_f-1} \|u(k)\|_2^2 \right] \\
    \mathrm{s.t.} & z_{k+1} = A_k z_k + B_k u_k\\
         & z_0 \sim \rho_0(z) dz,~z_{t_f} \sim \rho_1(z)dz.
  \end{array} \tag{P1} \label{eq:2}
\end{align}
Following a similar derivation as the continuous-time case studied in~\cite{Chen2017}, one can consider first solving the deterministic problem,
\begin{align}
  \begin{array}{ll}
    \min &  \sum_{k=0}^{t_f-1} \|u(k)\|_2^2  \\
    \mathrm{s.t.} & z_{k+1} = A_k z_k + B_k u_k\\
         & z_0= x,~z_{t_f} = y.
  \end{array} \tag{P2} \label{eq:4}
\end{align}
The solution to \eqref{eq:4} is given in closed form as 
\begin{align}
 \begin{split}
  C^*(x,y)&= \left( y- \Phi(t_f,0) x\right)^T\\
   &\hspace{2em}\cdot W_c(t_f,0)^{-1} \left( y- \Phi(t_f,0) x\right),
  \end{split}\label{eq:5}\\
  u^*(k) &= B_k^T \Phi(t_f,k+1)^T  W_c(t_f,0)^{-1}  \left(y - \Phi(t_f,0) x\right),
\end{align}
where
\begin{align}
  &\Phi(t_f,k) = A_{t_f-1}A_{t_f-2}\cdots A_k,
\end{align}
and
\begin{align}
  &W_c(t_f,0) = \sum_{k=0}^{t_f-1} \Phi(t_f,k+1)B_kB_k^T \Phi(t_f,k+1)^T.\label{eq:28}
\end{align}
Substituting the optimal cost \eqref{eq:5} into \eqref{eq:2} and applying the coordinate transformation
\begin{align}
  \begin{bmatrix}
    x \\ y 
  \end{bmatrix} \mapsto
  \begin{bmatrix}
    W_c(t_f,0) ^{-1/2} \Phi(t_f,0) x\\ 
    W_c(t_f,0)^{-1/2} y
  \end{bmatrix} \triangleq
  \begin{bmatrix}
    \hat{x} \\ \hat{y}
  \end{bmatrix},\label{eq:30}
\end{align}
transforms~\eqref{eq:5} into $\|\hat{x} - \hat{y}\|_2^2$, and so we arrive at the Kantorovich optimal transport problem
\begin{align}
  \begin{array}{ll}
    \min_{\hat{\pi}} & \int_{\mathbb{R}^n\times\mathbb{R}^n} \| \hat{x} - \hat{y} \|_2^2 d\hat{\pi}(\hat{x},\hat{y}) \\
    \mathrm{s.t.} & \hat{\pi}(\hat{x},\cdot) = \hat{\rho}_0(\hat{x})d\hat{x},~ \hat{\pi}(\cdot,\hat{y}) = \hat{\rho}_1(\hat{y})d\hat{y},
  \end{array} \tag{P3} \label{eq:6} 
\end{align}
with the distributions in the $(\hat{x},\hat{y})$ coordinates given by
\begin{align}
  &\resizebox{\columnwidth}{!}{$
  \hat{\rho}_0 (\hat{x}) = |W_c(t_f,0)|^{1/2} | \Phi(t_f,0) |^{-1} \rho_0\left(\Phi(t_f,0)^{-1} W_c(t_f,0)^{1/2} \hat{x}\right)$}\\
  &\hat{\rho}_1 (\hat{y}) = |W_c(t_f,0)|^{1/2}  \rho_0\left( W_c(t_f,0)^{1/2} \hat{y}\right).
\end{align}
Now, suppose that $\hat{T}$ is the Monge map that corresponds to the solution of \eqref{eq:6}.
Then, by using~\eqref{eq:30}, the solution to the original problem \eqref{eq:2} can be written in terms of its Monge map
\begin{align}
  y = T(x) = W_c(t_f,0)^{1/2} \hat{T} \left( W_c(t_f,0)^{-1/2} \Phi(t_f,0) x\right).
\end{align}
The  
optimal controls are thus given by,
\begin{equation}
  \resizebox{\columnwidth}{!}{
$u(k,x(k)) = B^T \Phi(t_f,k+1)^T W_c(t_f,0)^{-1} \left[ T(x) - \Phi(t_f,0) x \right]$.
}
\end{equation}

\subsection{Discrete-Time Optimal Transport --  Linear-Quadratic Case}
\label{sec:discr-time-optim}

In this subsection, we consider the more general case of a linear-quadratic cost function.
The problem is formulated as
\begin{align}
    \begin{array}{ll}
      \min &\mathbb{E}\left[ \sum_{k=0}^{t_f-1} \left\{\|u_k\|_{R_k}^2 + \|z_k-y\|_{Q_k}\right\}^2\right]\\
      \mathrm{s.t.} & z_{k+1} = A_kz_k + B_k u_k\\
               & z_0 \sim \rho_0(z)dz,~z_{t_f} = y \sim \rho_1(z)dz.\tag{P3}\label{eq:24}
    \end{array}
\end{align}

We proceed using the methodology in Remark~\ref{rem:1} by considering the solution to the deterministic problem,
\begin{align}
    \begin{array}{ll}
      \min & \sum_{k=0}^{t_f-1} \left\{\|u_k\|_{R_k}^2 + \|z_k-y\|_{Q_k}^2\right\}\\
      \mathrm{s.t.} & z_{k+1} = A_kz_k + B_k u_k\\
               & z_0 =x,~z_{t_f} =y.
    \end{array}\tag{P4}\label{eq:7}
\end{align}
We summarize the cost-to-go of \eqref{eq:7} in the following lemma.
Note that we utilize a pseudoinverse, present in \eqref{eq:3}.
While at first glance it may seem that this pseudoinverse severely limits the applicability of this lemma, this is not the case.
We discuss in Remark~\ref{rem:pseudo} (after the proof of the lemma) why the pseudoinverse is well-defined for all controllable LTV systems, and we highlight an example in \S\ref{sec:examples} that shows that the pseudoinverse is indeed well-behaved, even in pathological cases.
\vspace{0.5em}
\begin{lemma}
  \label{lem:1}
  The optimal cost-to-go of \eqref{eq:7} is quadratic in $x,y$, in that
  \begin{align}
      C^*(x,y) = x^TQ_x x + y^TQ_y y + 2 x^T Q_{xy} y,\label{eq:23}
  \end{align}
  where $Q_{xy}$ is invertible.
  The optimal control of~\eqref{eq:7} is given by $U^*:=\mathrm{vec}(\{u_k^*\}_{k=0}^{t_f-1})$,
  \begin{align}
    \begin{split}
      U^*&= K^*(y - \Phi(t_f,0)x)\label{eq:22}\\& - \Gamma_{U_1}P^{-1}A_{U_1}^T\tilde{Q}\left(\Omega x - (\mathbf{1}_{t_f}\otimes I_n)y\right),
      \end{split}\\
K^* &= \left(I - \Gamma_{U_1} P^{-1} \Gamma_{U_1}^T\tilde{R} - \Gamma_{U_1} P^{-1} A_{U_1}^T \tilde{Q}\Psi\right) \Gamma_y,
  \end{align}
  where the matrices comprising the optimal cost and control are given by
  \begin{align}
    Q_x &= K_1^T \tilde{Q} K_1 + K_3^T \tilde{R} K_3\\
    Q_y &= K_2^T \tilde{Q} K_2 + K_4^T \tilde{R} K_4\\
    Q_{xy} &= K_1^T\tilde{Q}K_2 + K_3^T \tilde{R} K_4\\
    K_1&=  (I-A_{U_1}P^{-1}A_{U_1}^T\tilde{Q})A_x-A_{U_1}P^{-1}\Gamma_{U_1}^T\tilde{R}\Gamma_x \\
    K_2&= (I-A_{U_1}P^{-1}A_{U_1}^T\tilde{Q})A_y -A_{U_1}P^{-1}\Gamma_{U_1}^T\tilde{R}\Gamma_y\\
    K_3&= (I-\Gamma_{U_1}P^{-1}\Gamma_{U_1}^T\tilde{R}) \Gamma_x -\Gamma_{U_1}P^{-1}A_{U_1}^T\tilde{Q} A_x\\
    K_4&= (I-\Gamma_{U_1}P^{-1}\Gamma_{U_1}^T\tilde{R}) \Gamma_y -\Gamma_{U_1}P^{-1}A_{U_1}^T\tilde{Q} A_y\\
    P&= A_{U_1}^T \tilde{Q} A_{U_1} + \Gamma_{U_1}^T \tilde{R} \Gamma_{U_1}\\
    A_x &= \Omega + \Psi  \Gamma_x,~
    A_y = \Psi \Gamma_y - \mathbf{1}_{t_f} \otimes I_n,~
    A_{U_1} = \Psi \Gamma_{U_1}\\
    \Gamma_x &= \begin{bmatrix}
                    \mathbf{0}\\ -S_2^\dag \Phi(t_f,0)
                \end{bmatrix},~
    \Gamma_y = \begin{bmatrix}
                    \mathbf{0} \\ S_2^\dag
                \end{bmatrix},~
    \Gamma_{U_1}= \begin{bmatrix}
                       I_{(t_f-n)m}\\ - S_2^\dag S_1
                   \end{bmatrix},\\
    \tilde{Q} &= \bigoplus_{k=0}^{t_f-1}Q_k,~\tilde{R} = \bigoplus_{k=0}^{t_f-1}R_k,
  \end{align}
where $S_1$, and $S_2$ are
\begin{align}
    &S_1 =\\ 
&  \begin{bmatrix}
    \Phi(t_f,1) B_0 & \Phi(t_f,2)B_1 & \cdots & \Phi(t_f,t_f-n) B_{t_f-n-1}
  \end{bmatrix},\\
&  S_2 = 
   \begin{bmatrix}
     \Phi(t_f,t_f-n+1) B_{t_f-n} & \cdots & B_{t_f-1}
   \end{bmatrix},
\end{align}
and
the matrices defined by the dynamics are given by
\begin{align}
&  \Psi = 
  \begin{bmatrix}
    \bar{\Upsilon}(0) \\
    \bar{\Upsilon}(1) \\ 
    \vdots \\ 
    \bar{\Upsilon}(t_f -1)
  \end{bmatrix},~
  \Omega = 
  \begin{bmatrix}
    \Phi(1,0) \\ 
    \Phi(2,0)\\
    \vdots \\
    \Phi(t_f,0)
  \end{bmatrix}\\
&  \Upsilon (l,0)= 
  \begin{bmatrix}
    \Phi(l,1) B_0 & \Phi(l,2) B_1 & \cdots & B_{l-1}
  \end{bmatrix}\\
&  \bar{\Upsilon}(l) :=
  \begin{bmatrix}
    \Upsilon(l,0) & | & \mathbf{0} & \cdots & \mathbf{0}
  \end{bmatrix} \in \mathbb{R}^{n \times m t_f}.
\end{align}

\end{lemma}
\begin{remark}    
For an LTI system, these matrices are simply
\begin{align}
    \Psi &= \begin{bmatrix}
               \mathbf{0} & \mathbf{0} & \cdots&&  \mathbf{0}\\
               B & \mathbf{0} & \\
               AB & B &   \ddots && \vdots \\
               \vdots & \vdots & \ddots & \ddots\\
               A^{t_f-2}B & A^{t_f-3}B & \cdots & B & \mathbf{0}
           \end{bmatrix},~
  \Omega = \begin{bmatrix}
               I \\ A & \\ A^2 \\ \vdots \\ A^{t_f-1}
            \end{bmatrix}
\end{align}
\hfill\(\triangle\)
\end{remark}

We now prove the lemma.

\begin{proof}  An analysis in the simpler case of LTI systems with $y=0$ may be found in~\cite{feller2017relaxed}.
  By setting $U(l) = \mathrm{vec}(\{u_k\})_{k=0}^{l},~U=U(t_f)$, and $\tilde{R} = \bigoplus_{k=0}^{t_f-1}R_k$, the control cost term can be written as
  \begin{align}
   \sum_{k=0}^{t_f-1} \|u_k\|_{R_k}^2 =    U^T \tilde{R} U.\label{eq:19}
  \end{align}
Similarly, by writing $ z_k = \Phi(k,0) x + \Upsilon(k,0) U(k)$ and $\tilde{Q} = \bigoplus_{k=0}^{t_f-1}Q_k$, the state cost term can be written as
\begin{align}
&  \sum_{k=0}^{t_f-1} \|z_k-y\|_{Q_k}^2 =\\ 
&    \left(\Omega x + \Psi U - (\mathbf{1}_{t_f} \otimes I_n) y\right)^T\tilde{Q}\left(\Omega x + \Psi U - (\mathbf{1}_{t_f} \otimes I_n) y\right).\label{eq:18}
\end{align}
To eliminate the equality constraints, we can write,
\begin{align}
  y &= \Phi(t_f,0)x + 
\begin{bmatrix}
  S_1 & S_2
\end{bmatrix}
\begin{bmatrix}
  U_1\\U_2
\end{bmatrix},
\end{align}
and we can thus parameterize $U$ as follows:
\begin{align}
 U &=
  \begin{bmatrix}
    U_1 \\ U_2
  \end{bmatrix} = 
  \begin{bmatrix}
    U_1 \\ S_2^\dag \left( y - \Phi(t_f,0)x - S_1U_1 \right)
  \end{bmatrix}\label{eq:3}\\& = 
  \begin{bmatrix}
    I_{(t_f-n)m} \\ -S_2^\dag S_1
  \end{bmatrix} U_1  + 
  \begin{bmatrix}
    \mathbf{0} \\ -S_2^\dag \Phi(t_f,0)
  \end{bmatrix} x + 
  \begin{bmatrix}
    \mathbf{0} \\ S_2^\dag
  \end{bmatrix}y \\
&= \Gamma_{U_1} U_1 + \Gamma_x x + \Gamma _y y.\label{eq:20}
\end{align}
Substituting \eqref{eq:20} into \eqref{eq:19} and \eqref{eq:18} yields the total cost as
\begin{align}
  \begin{split}
 J(x,y,U_1)&= (N_Q + A_{U_1}U_1)^T \tilde{Q} (N_Q + A_{U_1}U_1)\\& + (N_R + \Gamma_{U_1}U_1)^T \tilde{R} (N_R + \Gamma_{U_1}U_1)
 \end{split}\label{eq:21}\\
  N_Q&= A_x x + A_y y\\
  N_R&= \Gamma_x x + \Gamma_y y.
\end{align}
Taking the gradient of \eqref{eq:21} with respect to $U_1$, setting it to zero, and solving for $U_1^*$ yields,
\begin{align}
 U_1^* &= -P^{-1} (A_{U_1}^T \tilde{Q} N_Q + \Gamma_{U_1}^T \tilde{R} N_R ).\label{eq:31}
\end{align}
Substituting this form of $U_1^*$ into \eqref{eq:20} yields the optimal control as in~\eqref{eq:22}, and substituting $U_1^*$ into \eqref{eq:21} yields the optimal cost as in~\eqref{eq:23}.
It can be checked that $Q_{xy}$ is positive definite.\end{proof}

\begin{remark}
\label{rem:pseudo}
Our technical assumption in the above lemma is that the LTV system is controllable, in the sense that the controllability Gramian is positive definite.
The matrix $S_2$ in \eqref{eq:3} may be zero, even in cases when the underlying system is controllable.
The pseudoinverse is well-defined in this case due to the elementary property $\mathbf{0}_{m\times n}^\dag = \mathbf{0}_{n\times m}$.
See the example in \S\protect\ref{sec:examples} for more details.
\hfill\(\triangle\)
\end{remark}

We can now state and prove the main theorem regarding the solution of Problem \eqref{eq:24}.
\begin{theorem}
  \label{thr:1}
    Consider the setting of Problem~\eqref{eq:24}, and the Kantorovich optimal transport problem,
    \begin{align}
        \begin{array}{ll}
            \min_{\pi} &  \int_{\mathbb{R}^n\times\mathbb{R}^n} C^*(x,y)d\pi(x,y)\\
          \mathrm{s.t.} & x\sim\rho_0(x)dx,~y\sim\rho_1(y)dy
        \end{array},\label{eq:25}
    \end{align}
    where $C^*(x,y)$ is given by~\eqref{eq:23}.
    Then, the optimal coupling $\pi^*$ of \eqref{eq:25} is given by
    \begin{align}
        \pi^*(x,y) = (\mathrm{Id}\times T^*)_{\#} \rho_0(x),
    \end{align}
    where $T^*_{\#}\rho_0 = \rho_1$.
    Furthermore, the control inputs optimizing Problem~\eqref{eq:24} is given by,
    \begin{align}
        U& =  K^*(T^*(x) - \Phi(t_f,0)x)\\& - \Gamma_{U_1}P^{-1}A_{U_1}^T\tilde{Q}\left(\Omega x - (\mathbf{1}_{t_f}\otimes I_n)T^*(x)\right),\label{eq:26}
    \end{align}
    where the relevant matrices are defined in Lemma~\ref{lem:1}.
\end{theorem}
\begin{proof}
    We can see that the Kantorovich problem~\eqref{eq:25} corresponds to Problem~\eqref{eq:24}.
    By Theorem~2.2 in~\cite{Hindawi2011}, the solution to Problem~\eqref{eq:25} under the cost $C^*(x,y)$ from~\eqref{eq:23} exists, and is of the form
    \begin{align}
        \pi^*(x,y) &= (\mathrm{Id}\times T^*)_{\#} \rho_0(x),
    \end{align}
    with
    \begin{align}
      T^*(x) &= -\frac{1}{2} Q_{xy}^{-1} \nabla \varphi(x),
    \end{align}
    where $\varphi(x)$ is a convex function.
    From the calculation in Lemma~\ref{lem:1}, the optimal control in~\eqref{eq:26} follows.
\end{proof}

\begin{remark}
    The solution to Problem~\eqref{eq:25} yields a Monge map $T^*$ that transports $x\sim\rho_0(x)dx$ to $y:=T^*(x) \sim \rho_1(y)dy$, minimizing the expected cost-to-go from $x$ to $y$.
    Another interpretation of this map is that it pairs initial and terminal states $(x,y)$ in such a manner that it minimizes the LQR cost averaged over the distribution of initial states.
    We exploit this interpretation in \S\ref{sec:examples}, where we discuss an application to swarm deployment.
\hfill\(\triangle\)
\end{remark}

First, we examine the numerical computation of $T^*(x)$.

\subsection{Numerical Computation of the Monge Map}

In general, the Monge map $T(x)$ is difficult to compute numerically~\cite{Peyre2011,Peyre2019}.
In fact, the optimal control formulation~\eqref{eq:11} was devised by Brenier and Benamou precisely to numerically compute $T(x)$, and devising fast solvers for this problem is an area of active research~\cite{Papadakis2014}. 
In one dimension, a classical result (used in~\cite{Chen2017}) determines the Monge map in terms of the cumulative distribution functions of the initial and terminal densities as
\begin{align}
    \int_{-\infty}^x \rho_0(x)dx = \int_{-\infty}^{T(x)}\rho_1(y)dy.
\end{align}
This can readily be solved to high precision with a bisection algorithm.

For systems with $n>1$ states, the situation is more complicated.
For example, in the single-integrator system $x_{k+1} = x_k + u_k$, the one-timestep Monge map exists explicitly when the initial and termimal distributions are Gaussian.
Suppose $\rho_0,\rho_1$ are,
$
    \rho_0(x) \sim \mathcal{N}(m_0,\Sigma_0),$ $\rho_1(x)\sim\mathcal{N}(m_1,\Sigma_1).
$
Then, the optimal Monge map is a shift and scaling~\cite{Knott1984},
$
  T(x) = \tilde{A}(x-m_0) + m_1,
$
with
$
  \tilde{A} = \Sigma_0^{-1/2} \left( \Sigma_0^{1/2} \Sigma_1 \Sigma_0^{1/2}\right)^{1/2} \Sigma_0^{-1/2}.
$

For general distributions, we outline a discretization-based method for computing $\pi^*$ from~\eqref{eq:9}, and then generating the image of $T^*(x)$ from this approximate $\pi^*$.
Suppose we discretize $\mathcal{X} = \mathcal{Y} := \mathbb{R}^n$ into cells $\{X_i\}_{i=1}^{n_x}$, $\{Y_j\}_{j=1}^{n_y}$, and then define probability mass vectors $\bmrho_0\in\mathbb{R}^{n_x}$, $\bmrho_1\in\mathbb{R}^{n_y}$ representing $\rho_0$, $\rho_1$, as
\begin{align}
    \bmrho_i = \int_{X_i} \rho_0(x)dx,~\bmrho_j = \int_{Y_j} \rho_1(y)dy.
\end{align}
The cost in \eqref{eq:9} can be written over this discrete space as
\begin{align}
  \int_{\mathbb{R}^n \times \mathbb{R}^n} C(x,y) d\pi(x,y) \longrightarrow
  \sum_{i,j} C(x_i,y_j) \bmpi(x_i,y_j),
\end{align}
where $x_i,y_i$ are representative coordinates of the cell, say their centroids.
The marginal constraints can be imposed on $\bmpi$ as
\begin{align}
    \bmpi \mathbf{1}_{n_y} = \bmrho_0,~\bmpi^T \mathbf{1}_{n_x} = \bmrho_1. 
\end{align}
Letting $C(x_i,y_j) := C_{ij}$, and $\bmpi(x_i,y_j):=\bmpi_{ij}$, we arrive at the linear program
\begin{align}
    \begin{array}{ll}
        \min_{\bmpi} & \sum_{i,j} C_{ij} \bmpi_{ij}\\
      \mathrm{s.t.} & \bmpi \mathbf{1}_{n_y} = \bmrho_0,~\bmpi^T \mathbf{1}_{n_x} = \bmrho_1\\
                     &\bmpi_{ij} \geq 0,~\forall i,j.
    \end{array}\label{eq:27}
\end{align}
To recover a discrete image of the map $T$, one has to numerically `un-do' the pushforward operation that $\bmpi$ represents.
This is done by the element-wise division of $\bmpi(x_i,y_j)$ by $\bmrho_0(x_i)$: 
\begin{align}
    \bmtau(x_i,y_j) := \dfrac{\bmpi(x_i,y_j)}{\bmrho_0(x_i)}.
\end{align}
Note that this definition requires that $\bmrho_0$ must be strictly positive over the discrete domain; alternately  if $\bmrho_0(x_i)=0$, then the corresponding row of $\bmpi$ must also be 0 from the constraints in~\eqref{eq:27}. 
In this case, we can define $\bmtau(x_i,y_j)$ arbitrarily, since if there is no mass to move from $x_i$, it is irrelevant where that mass should move to.
Note that Problem~\eqref{eq:27} suffers from the `curse of dimensionality' due to the discretization of $\mathbb{R}^n$.
Fast approximations of optimal transport are an ongoing area of research, and in the near future one may expect that Problem~\eqref{eq:27}, or approximations of it, may be computationally tractable for large state-spaces~\cite{weedarxiv,Cuturi2013,weedlight}.

The graph of $T$ over $\{x_i\}_{i=1}^k$ can then be determined by applying the map $\bmtau$ to the domain $\{x_i\}_{i=1}^k$.
Suppose $[X_1,\dots,X_n]$ are appropriately-vectorized coordinates in each of the $n$ directions of the discretized domain in $\mathbb{R}^n$.
Then, the matrix $\bmtau$ generates the image of $T$ as follows:
\begin{align}
    \bmtau [X_1,\dots,X_n] = [T_1,\dots,T_n],
\end{align}
where $T_i$ is the vectorized map over the domain in the $i$th direction of $\mathbb{R}^n$.

\section{Examples}
\label{sec:examples}

\begin{figure}
    \centering
    \includegraphics[width=\columnwidth]{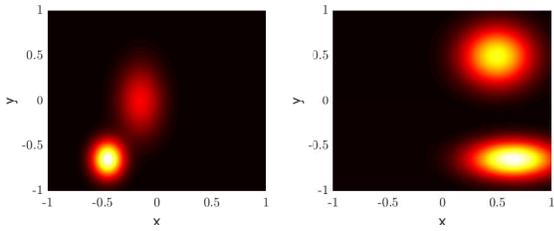}
    \caption{True distribitions of the initial (left: $\rho_0(x)$) and target (right: $\rho_1(x)$) states.}
    \label{fig:eg1dist}
\end{figure}

In this section, we provide numerical experiments of the results in \S\ref{sec:results}.
Code (and parameters) for the examples can be accessed at~\cite{2021CDCCode}.
The runtimes for the computation of the optimal transport maps are 12s and 9s, for examples \ref{sec:2d-lqr-example} and \ref{sec:swarm-deployment-2d} respectively, on an Intel Core i7-9700K CPU (3.60GHz).

\subsection{2D LQR Example on LTI System}
\label{sec:2d-lqr-example}
First, we provide an example of the numerical implementation of Theorem~\ref{thr:1}.
We implemented the optimal transport method for LQR on a 2-state, 1-input system with matrices
\begin{align}
    A = \begin{bmatrix}
   0.9& -0.1 \\
         -0.1 & 0.8          
        \end{bmatrix},~ B = 
  \begin{bmatrix}
      1 \\ 0
  \end{bmatrix}\\
  Q_k = I_2,~R_k=1,~0\leq k \leq 10.
\end{align}
Let the states be denoted by $z_k:=(z_k^{(1)},z_k^{(2)})$.
Our initial states were distributed according to $\rho_0(x)$ depicted in Fig.~\ref{fig:eg1dist}, and we sought to steer the system to the distribution $\rho_1(y)$, also depicted in Fig.~\ref{fig:eg1dist}.

The distributions $\rho_0,\rho_1$ are supported on a discrete grid on the cube $[-1,1]^2$ with a discretization length $\Delta x = \Delta y = 0.0588$.
The optimal transport map $T^*$ was computed by solving \eqref{eq:27} discretized on this grid, with a cost matrix $C$ computed using \eqref{eq:23} from Lemma~\ref{lem:1}.
In Fig.~\ref{fig:eg1map}, we show a color map of the image of $T^*$.
On the left is the coordinate in the first dimension as a function of the initial $(x^{(1)},x^{(2)}) = (z_{0}^{(1)},z_0^{(2)})$, and on the right is the coordinate in the second dimension as a function of the initial $(z_{0}^{(1)},z_0^{(2)})$.
One can note that the numerical approximation contains outliers in regions where $\bmrho_0$ has little mass.

We ran an experiment with 10,000 i.i.d.~random initial conditions sampled from $\rho_0$.
For each initial condition $z_0$, the optimal map computed the corresponding final condition as $z_{t_f} = T^*(z_0)$.
The simulation then used the optimal control inputs~\eqref{eq:26} to guide the system to $T^*(z_0)$ over a time horizon of $0\leq k \leq 10$.
Plots of the empirical distributions are shown in Fig.~\ref{fig:eg1emp} at times $t=0,3,7,10$.

\begin{figure}
    \centering
    \includegraphics[width=\columnwidth]{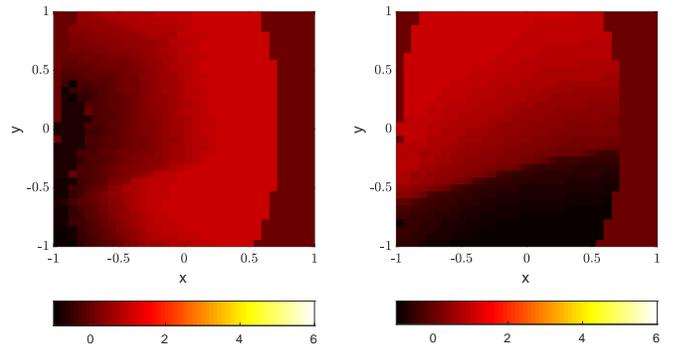}
    \caption{Images of the optimal map $T^*(x)$ in the $y^{(1)}$ (left) and $y^{(2)}$ (right) coordinates of the target domain.}
    \label{fig:eg1map}
\end{figure}

\begin{figure}
    \centering
    \includegraphics[width=\columnwidth]{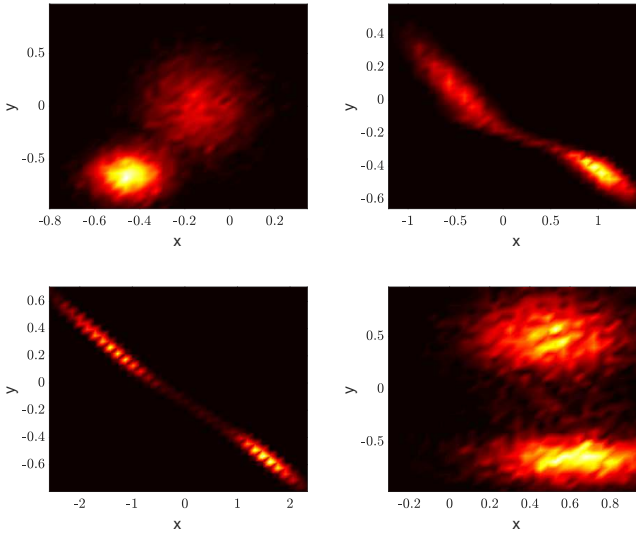}
    \caption{Empirical distributions of the states of the system over time. Top left: $t=0$. Top right: $t=3$. Bottom left: $t=7$. Bottom right: $t=10$.}
    \label{fig:eg1emp}
\end{figure}

\subsection{Swarm Deployment - 2D LQR on an LTV System}
\label{sec:swarm-deployment-2d}
Consider the task of assigning target positions to $n$ agents whose initial states have an empirical distribution approximating $\rho_0(x)dx$, but not randomly instantiated.
For example, consider $n$ agents spaced at constant intervals in the cube $[-1,1]^2$, as depicted in the middle-left subfigure of Fig.~\ref{fig:eg2dist}.
Clearly, this is an approximation of a uniform distribution.
Our target distribution is the logo of the Swiss Federal Institute of Technology, Z\"{u}rich, discretized over a $35\times 35$ pixel domain.
A target application could be a swarm of UAVs providing a background performance act during a university event.

Using LTV discrete single-integrator dynamics
\begin{align}
\begin{split}
  &A_k = Q_k = R_k= I_2,~0\leq k \leq 10,\\
  &B_k = I_2,~0\leq k \leq 5,~B_k = \mathbf{0}_{2\times 2},~5<k\leq 10,
\end{split}\label{eq:29}
\end{align}
we compute the optimal map using~\eqref{eq:27} with the cost matrix~\eqref{eq:23} from Lemma~\ref{lem:1}, depicted in Figure~\ref{fig:eg2map}.
The simulation was again produced over a time horizon of $0\leq k \leq 10$.
This time, we plot the explicit mapping between points in a grid and their target states as generated by the map $T^*$, as shown in the bottom-right of Figure~\ref{fig:eg2map}.

The dynamics~\eqref{eq:29} are controllable, in the sense that the controllability Gramian~\eqref{eq:28} $W_c(t_f,0)$ is positive-definite, however the matrix $S_2$ in \eqref{eq:3} is $\mathbf{0}_{4\times2}$.
Since $\mathbf{0}_{4\times 2}^\dag = \mathbf{0}_{2\times 4}$, by \eqref{eq:3}, this simply means that the control is zero for $6\leq k \leq 10$.
As the system is controllable, it is steered to the final position by timestep $k=5$, as evident in Figure~\ref{fig:lqr}.

\begin{figure}
  \centering
  \includegraphics[width=\columnwidth]{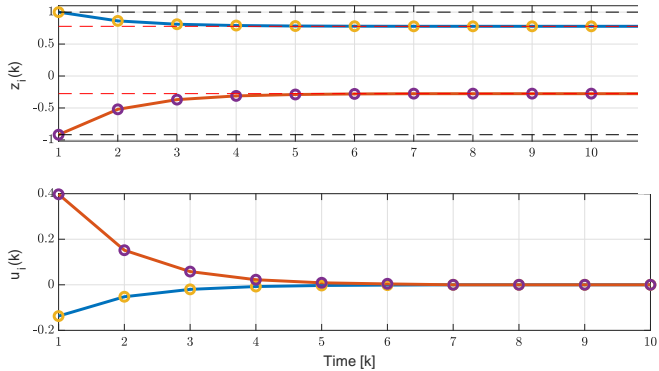}
  \caption{Top: Trajectory of dynamics \protect\eqref{eq:29}. Bottom: Control computed by solving \protect\eqref{eq:24} using \texttt{cvx}\protect\cite{cvx} (solid line) and via \protect\eqref{eq:22} (markers). Red (black) dashed lines indicate target (initial) states.}
  \label{fig:lqr}
\end{figure}

\begin{figure}
    \centering
    \includegraphics[width=\columnwidth]{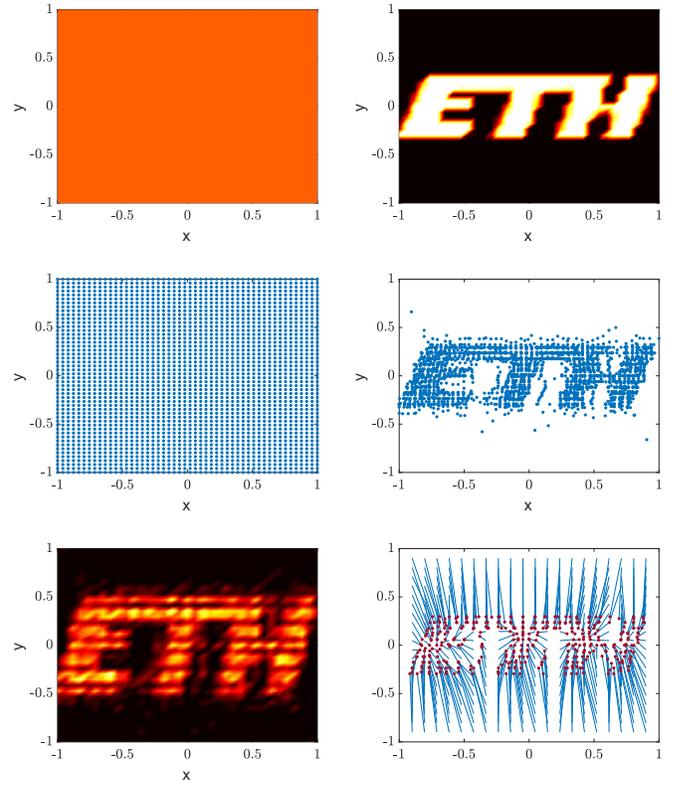}
    \caption{Top: plots of the initial uniform distribution (left), and target distribution representing the ETH logo (right). 
Middle: Initial conditions $(x,y)$ uniformly spaced in $[-1,1]^2$ (left), and their corresponding terminal conditions $T^*(x,y)$ (right).
Bottom: empirical distribution of the terminal states from the middle-right, and a plot showing the corresponding final states interpolated from an initial state.}
    \label{fig:eg2dist}
\end{figure}

\begin{figure}
    \centering
    \includegraphics[width=\columnwidth]{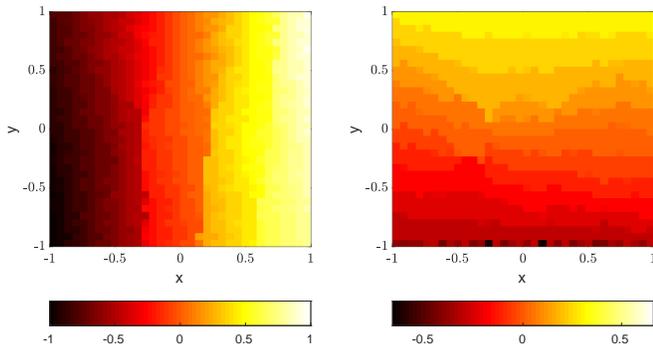}
    \caption{Images of the optimal map $T^*(x)$ in the $x$ (left) and $y$ (right) coordinates of the target domain. }
    \label{fig:eg2map}
\end{figure}


\section{Conclusion}
\label{sec:conclusion}

In this paper, we studied the discrete-time linear-quadratic regulator with uncertainties in the initial state, and how optimal transport can be used to guide the system to a final state with an uncertainty specified by a target probability density.
We derived the form of the optimal control from the optimal transport map, and discussed numerical implementations of this map.
Finally, we provided numerical examples with an application to swarm deployment.

Future work may include online computation of the optimal transport map corresponding to the LQ cost, and studying systems where additional uncertainty comes from being driven by noise of arbitrary distributions.




\end{document}